\documentclass[11pt,reqno]{amsart}
\usepackage{graphics}
\usepackage{amsmath}
\usepackage{amsthm}
\usepackage{amssymb}
\usepackage{verbatim}

\theoremstyle{plain}                    

\theoremstyle{plain}
\newtheorem{theorem}{Theorem}[section]
\newtheorem{lemma}[theorem]{Lemma}
\newtheorem{proposition}[theorem]{Proposition}
\newtheorem{corollary}[theorem]{Corollary}

\theoremstyle{definition}

\theoremstyle{remark}
\newtheorem{remark}[theorem]{Remark}

\DeclareMathOperator{\zero}{\mathcal{Z}}
\DeclareMathOperator{\rec}{Rec}

\newcommand{\NN}{\mathbb{N}} 
\newcommand{\holder}{H\"{o}lder }
\newcommand{\eps}{\epsilon}
\newcommand{\kz}{F}
\newcommand{\kzz}{V}
\newcommand{\di}{\mathcal{G}}

\DeclareMathOperator{\sign}{sign}

\begin{document}

\title{Isolated zeros for Brownian motion with variable drift}

\address{Ton\'{c}i Antunovi\'{c}\\
University of California, Berkeley\\
Department of Mathematics\\
Berkeley, CA 94720
}
\email{tantun@math.berkeley.edu}

\address{Krzysztof Burdzy\\
University of Washington\\
Department of Mathematics\\
Seattle, WA 98195}
\email{burdzy@math.washington.edu}

\address{Yuval Peres\\
Microsoft Research\\
Theory Group\\
Redmond, WA 98052}
\email{peres@microsoft.com}

\address{Julia Ruscher\\
Fachbereich Mathematik, Sekr. MA 7-4\\
Technische Universit\"{a}t Berlin\\
Strasse des 17. Juni 136\\
D-10623 Berlin
}
\email{ruscher@math.tu-berlin.de}

\author[T. Antunovi\'{c}]{Ton\'{c}i Antunovi\'{c}}

\author[K. Burdzy]{Krzysztof Burdzy}

\author[Y. Peres]{Yuval Peres}

\author[J. Ruscher]{Julia Ruscher}


\maketitle

\begin{abstract}
It is well known that standard one-dimensional Brownian motion $B(t)$ has no isolated zeros almost surely. 
We show that for any $\alpha<1/2$ there are $\alpha$-H\"{o}lder continuous functions $f$ for which the process $B-f$ has isolated zeros with positive probability.
We also prove that for any continuous function $f$, the zero set of $B-f$ has Hausdorff dimension at least $1/2$ with positive probability, and $1/2$ is an upper bound on the Hausdorff dimension if $f$ is $1/2$-H\"{o}lder continuous or of bounded variation.
\end{abstract}

\section{Introduction}
Let $B$ be standard one-dimensional Brownian motion and $f \colon I \to \mathbb{R}$ a continuous function defined on some interval $I \subset \mathbb{R}^+$. A standard result is that the zero set of $B$ has no isolated points almost surely, see Theorem 2.28 in \cite{MP}. By the Cameron-Martin theorem (see Theorem 1.38 in \cite{MP} or Theorem 2.2 in Chapter 8 in \cite{RY}) the zero set of the process $B-f$ has no isolated points almost surely if $f$ is in the Cameron-Martin space $\mathbf{D}(I)$ (integrals of functions in $\mathbf{L}^2(I)$).
We will prove that the same is true for any function $f$ which is $1/2$-H\"{o}lder continuous. Since all functions in $\mathbf{D}(I)$ are $1/2$-H\"{o}lder continuous, this is a stronger statement than the one implied by the Cameron-Martin theorem.
For any function $g$ defined on some subset (or the whole) of $\mathbb{R}^+$ denote by $\zero(g)$ the set of zeros of $g$ in $(0,\infty)$. We remove the origin from consideration since the origin is an isolated zero of the process $B-f$ for any $f$ growing fast enough in the neighborhood of the origin, say $f(t) > t^{1/3}$.

\begin{proposition}
\label{thm:Holder-1/2}
For $f \colon \mathbb{R}^+ \to \mathbb{R}$ which is $1/2$-H\"{o}lder continuous on compact intervals, the set $\zero(B-f)$ has no isolated points almost surely.
\end{proposition}

The condition that $f$ is $1/2$-H\"{o}lder continuous is sharp in the following sense.

\begin{theorem}
\label{thm: isolated_zeros}
For every $\alpha<1/2$ there is an $\alpha$-\holder continuous function $f \colon \mathbb{R}^+ \to \mathbb{R}$ such that the set $\zero(B-f)$ has isolated points with positive probability.
\end{theorem}

Theorem \ref{thm: isolated_zeros} will follow directly from Proposition \ref{prop:isolated_zeros_first_example}.
An example of function $f$ satisfying
 Theorem \ref{thm: isolated_zeros} is given in Section \ref{Isolated zeros - first example}. 
For $\gamma<1/2$, let $C_\gamma$ denote the middle $(1-2\gamma)$-Cantor set 
and let $f_\gamma$ be the corresponding Cantor function, shifted to an interval away from the origin; see Figure \ref{fig:cantor} and Section \ref{Isolated zeros - first example} for a precise definition. 
For
 $\gamma<1/4$, the set $\zero(B-f_\gamma)$ 
has
 isolated zeros with positive probability, and all 
such zeros are
 contained in the Cantor set $C_\gamma$. The 
proof of this claim
 consists of constructing a subset of $C_\gamma$ which 
contains
 zeros of $B-f_\gamma$ with positive probability, and in which any zero 
is
 isolated. En route we obtain the following result of independent interest.

\begin{theorem}
  \label{prop: hitting characterisation}
Let $f_\gamma$ be a Cantor function. Then $\mathbb{P}(\zero(B-f_\gamma) \cap C_\gamma \neq \emptyset)>0$ if and only if $\gamma \neq 1/4$.
\end{theorem}

The
 case $\gamma=1/4$ of the above theorem has already been resolved by Taylor and Watson (see Example 3 in \cite{TaylorWatson}). Their interest in the graph of the restriction $f_\gamma|_{C_\gamma}$ stemmed from the fact that, although the projection of this set
on
 the vertical axis is an interval,
the graph of Brownian motion does not intersect this set almost surely.

\begin{figure}

\includegraphics{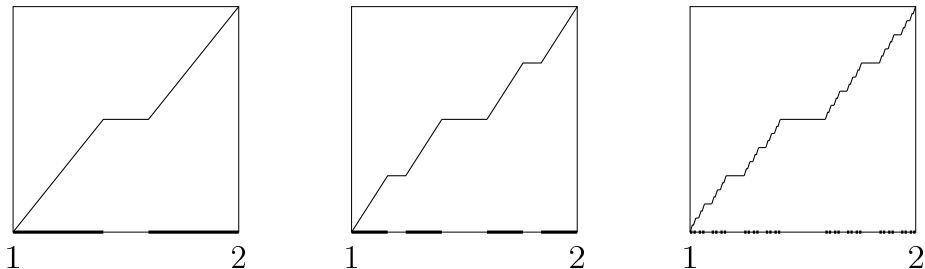}

\caption{Approximations of the Cantor function on the interval $[1,2]$ (functions $f_{\gamma,n}$ from the construction in Section \ref{Isolated zeros - first example}) for $\gamma=0.4$ and $n=1,2,5$. Approximations of the Cantor set (sets $C_{\gamma,n}$ from the construction in Section \ref{Isolated zeros - first example}) are drawn in bold.}
\label{fig:cantor}
\end{figure}



Part (ii) of Proposition \ref{prop:basic_facts} shows that isolated zeros of the process $B-f$ can occur only where the function $f$ increases or decreases very quickly. In the following theorem we bound the Hausdorff dimension of such sets. 
The Hausdorff dimension of a set $A \subset \mathbb{R}^+$ will be denoted $\dim(A)$. 

\begin{theorem}
\label{set A_f has dim less than 1/2 }
For any 
continuous function $f \colon \mathbb{R}^+ \to \mathbb{R}$ there exists a set $A_f$, such that $\dim(A_f) \leq 1/2$ and such that, almost surely, all isolated points of $\zero(B-f)$ are contained in $A_f$.
\end{theorem}

It is a classical result that the zero set of Brownian motion has Hausdorff dimension  $1/2$ almost surely, see Theorem 4.24 in \cite{MP}.
Of course, for any compact interval $I$ not containing $0$ and any continuous function $f \colon I \to \mathbb{R}$ the event $\{\zero(B-f)=\emptyset\}$ will have a non-zero probability, and it is easy to construct a function $f \colon \mathbb{R}^+ \to \mathbb{R}$ with the same property. However, we prove that adding a continuous drift can not decrease the Hausdorff dimension of the zero set almost surely. This is the content of the following theorem.

\begin{theorem}
\label{thm:hausdorff_lower_bound}
For 
continuous function $f \colon \mathbb{R}^+ \to \mathbb{R}$, the set $\zero(B-f)$ has Hausdorff dimension greater 
than or equal to
 $1/2$ with positive probability.
\end{theorem}

As the following example shows, upper bounds on the Hausdorff dimension of the zero set can not be obtained without additional assumptions on the drift $f$. 
Recall that fractional Brownian motion $B^{(H)} \colon \mathbb{R}^+ \to \mathbb{R}$ with Hurst index $0<H<1$ is a continuous, centered Gaussian process, such that $\mathbb{E}(|B^{(H)}(t)-B^{(H)}(s)|^2) = |t-s|^{2H}$.
Taking the drift $f$ to be an independent sample of fractional Brownian motion with Hurst index $H$, one gets that the Hausdorff dimension
$\dim(\zero(B-f))$ is bounded from below by $1-H$, 
 almost surely. This follows from the proof of the same fact for unperturbed fractional Brownian motion in Theorem 4 in Chapter 18 of \cite{Kahane} (see also Proposition 5.1 and Section 7.2 in \cite{BDG}).

\begin{theorem}
\label{thm:hausdorff_upper_bound}
Let $f \colon \mathbb{R}^+ \to \mathbb{R}$ be either $1/2$-H\"{o}lder continuous on compact intervals or of bounded variation on compact intervals. Then the Hausdorff dimension of $\zero(B-f)$ is at most
$1/2$,
 almost surely.
\end{theorem}

Arguments from the proof of Theorem \ref{thm:hausdorff_lower_bound} and Theorem \ref{thm:hausdorff_upper_bound} imply the following corollary.

\begin{corollary}\label{cor:sharp_haudorff_for_holder}
If $f \colon [0,1] \to \mathbb{R}$ is a $1/2$-H\"{o}lder continuous function, such that $f(0)=0$, then the Hausdorff dimension of $\zero(B-f)$ is equal to $1/2$, almost surely.
\end{corollary}


\begin{remark}
\label{rem:no_isolated_extremes}
For a function $g \colon \mathbb{R}^+ \to \mathbb{R}$, define $M_g(t)= \max_{0 \leq s \leq t}g(s)$ and 
denote the
 set of its record times 
by
 $\rec(g)=\{t>0:M_g(t)=g(t)\}$. For standard Brownian motion $B$, a result of L\'{e}vy says that the processes $(M_B(t)-B(t))_t$ and $(|B(t)|)_t$ have the same distribution (see e.g.~Theorem 2.34 in \cite{MP}), which implies that sets $\zero(B)$ and $\rec(B)$ have the same distribution (on the Borel sigma algebra of families of closed subsets of $\mathbb{R}^+$, generated by the Hausdorff metric). In general, for Brownian motion with drift there is no such correspondence. Actually, one can see that there are no isolated points in the set of record times of the process $B-f$ almost surely. 
This is proven in part (ii) of Proposition \ref{prop: zeros extremes and record times}.
\end{remark}

\subsection{Related results}
Connection between H\"{o}lder continuity of drift $f$ and path properties of $B-f$ has already been observed.
For $d \geq 2$, a function $f \colon \mathbb{R}^+ \to \mathbb{R}^d$ is called polar if, for $d$-dimensional Brownian motion $B$ started at the origin and any point $x \in \mathbb{R}^d$, the probability that there is a $t>0$ such that $B(t)-f(t)=x$ is positive. In \cite{Grav} Graversen constructed $\alpha$-H\"{o}lder continuous functions which are polar for two dimensional Brownian motion, for $\alpha <1/2$. In \cite{LeGall} Le Gall showed that $1/2$-H\"{o}lder continuous functions are not polar for two dimensional Brownian motion and that the same conclusion holds in higher dimensions when $f$ satisfies a slightly stronger condition than $1/d$-H\"{o}lder continuity.
In a recent paper \cite{APV} it was shown that for any $\alpha<1/d$ there are $\alpha$-H\"{o}lder continuous functions $f$ such that the image of $B-f$ covers an open set almost surely.

We will briefly review some related results on intersections of Brownian trajectories with non-smooth paths.  
Let $S_\alpha $ be the family of all functions $f: [0,1]\to \mathbb{R}$ such that $\sup_{0\leq t \leq 1} |f(t)| \leq 1$ and $\sup_{0\leq s,t \leq 1} |f(s) - f(t)|/ |s-t|^\alpha \leq 1$. Define the local time of Brownian motion $B$ on $f$ by the formula $L^f_t = \lim_{\varepsilon\downarrow 0} (1/2\varepsilon)\int_0^t {\bf 1}_{[f(s) -\varepsilon, f(s) + \varepsilon]}(B(s))ds$. It was proved in \cite{BB1,BB2} that the supremum over $f \in S_\alpha$ of $L^f_1$ is finite for $\alpha > 5/6$ and infinite for $\alpha < 1/2$. The function $(f,t) \to L^f_t$ is continuous over $S_\alpha \times [0,1]$ for $\alpha > 5/6$.  

Suppose that $g: \mathbb{R}_+ \to \mathbb{R}$ is a continuous function and let $X$ be Brownian motion reflected on $g$; see \cite{BCS} for a precise definition. Let $A_g$ be the set of all $t>0$ such that $\mathbb{P}(X(t) =g(t))>0$. It was proved in \cite{BCS} that for every continuous function $g$ we have $\dim(A_g) \leq 1/2$ and for some continuous functions $g$ we have $\dim(A_g) = 1/2$.

\bigskip

\section{Isolated zeros - general results}\label{Isolated zeros - general results}

For an interval $I$ we denote its length by $|I|$ and say it is \textit{dyadic} if it is of the form $I=[k2^{m},(k+1)2^m]$ for integers $k>0$ and $m$. 
For intervals $I$ and 
$J$, we will write
 $I < J$ if $J$ is located
to the right
 of $I$.

\begin{remark} \label{6.30.1}
We will repeatedly use the following simple observations. 

(i) Suppose that $F_k$, $k\geq 1$, are events and for some $p>0$ and all $k$ we have $\mathbb{P}(F_k) \geq p$. Recall that $\limsup_k F_k = \bigcap_{n\geq 1} \bigcup_{k\geq n} F_k$ is the event that infinitely many $F_k$'s occur. Then $\mathbb{P}(\limsup_k F_k) \geq p$.

(ii) As an easy consequence of the Cauchy-Schwarz inequality we have
 $\mathbb{P}(Z>0) \mathbb{E} (Z^2) \geq (\mathbb{E} Z)^2$, for any nonegative variable $Z$. See Lemma 3.23 in \cite{MP}.

\end{remark}





\begin{proposition}
\label{prop:basic_facts}
Let $f \colon \mathbb{R}^+ \to \mathbb{R}$ be a continuous function.
\begin{itemize}
\item[(i)] 
Let $A \subset \mathbb{R}^+$ be a set such that for any $t \in A$ there is an $\alpha < 1/2$ such that $\liminf_{s \to t}\frac{|f(s)-f(t)|}{|t-s|^\alpha} > 0$. Then, almost surely any point in $\zero(B-f) \cap A$ is isolated in $\zero(B-f)$.
\item[(ii)] 
Almost surely all isolated points of $\zero(B-f)$ are located inside the set $A_f^+ \cup A_f^-$, where 
$A_f^+ = \{t\in \mathbb{R}^+: \lim_{h \downarrow 0}\frac{f(t+h)-f(t)}{\sqrt{h}} = \infty\}$ and 
$A_f^- = \{t\in \mathbb{R}^+: \lim_{h \downarrow 0}\frac{f(t+h)-f(t)}{\sqrt{h}} = -\infty\}$.
\end{itemize}
\end{proposition}

\begin{proof}
(i)
First assume
 that the set $A$ is contained in $[0,N]$ for some large $N$. 
If a zero $s \in A \cap \zero(B-f)$ is not isolated,
then we can find a sequence $(s_n) \subset 
A \cap \zero(B-f)$, 
 converging to $s$, which, for some $\alpha <1/2$, necessarily satisfies 
\[\liminf_{n \to \infty}\frac{|f(s_n)-f(s)|}{|s_n-s|^\alpha} > 0 \ \text{ and } \  \liminf_{n \to \infty}\frac{|B(s_n)-B(s)|}{|s_n-s|^\alpha} > 0. \]
However, this is impossible since, by Levy's modulus of continuity, almost surely, 
there exists an $h_1>0$ such that for all $h\in(0,h_1)$ and all 
$0 \leq t \leq N$ we have $|B(t+h)-B(t)| \leq 3\sqrt{h \log(1/h)}$, see e.~g.~Theorem 1.14 in \cite{MP}.
If $A$ is unbounded, apply the above reasoning to $A_N = A \cap [0,N]$ and let $N$ go to infinity.

(ii)
 First define $\tau_q = \min\left\{ t\geq q : B(t)=f(t)\right\}$ and notice that any isolated zero of the process $B-f$ must equal $\tau_q$, for some $q \in \mathbb{Q}$. This is because, for any zero $s \in \zero(B-f)$, not of the form $\tau_q$, and a sequence of rational numbers $(q_n)$ converging to $s$ from below, we have $\lim_n \tau_{q_n} = s$.
Therefore, it is enough to prove that for each $q \in \mathbb{Q}^+$, the event that $\tau_q \notin A_f^+ \cup A_f^-$  and that $\tau_q$ is isolated in the set $\zero(B-f)$, has probability zero. 

Fix a positive integer $M$ and define sequences of functions 
\[s_n^-(t) = \max\{0 \leq h \leq 1/n: f(t+h)-f(t)\leq M\sqrt{h}\}\] 
and 
\[s_n^+(t) = \max\{0 \leq h \leq 1/n: f(t+h)-f(t)\geq -M\sqrt{h}\}.\] 
Since $f$ is continuous, it is easy to see that for each $n$, the functions $s_n^+$ and $s_n^-$ are measurable. Also define
\[
\overline{A}_f(M) = \Big\{ t : \liminf_{h \downarrow 0}\frac{f(t+h)-f(t)}{\sqrt{h}} < M, \limsup_{h \downarrow 0}\frac{f(t+h)-f(t)}{\sqrt{h}} > -M\Big\}.
\]
For all $t \in \overline{A}_f(M)$  it holds that $s_n^+(t)>0$ and $s_n^-(t)>0$, for all $n$.
Since $\tau_q$ is a stopping time, the process $B_q(t)=B(\tau_q+t)-B(\tau_q)$,  is, by the strong Markov property, a Brownian motion independent of the sigma algebra $\mathcal{F}_{\tau_q}$. 
 Let $F_-$ denote the event that $B_q(s_n^-(\tau_q)) \geq M\sqrt{s_n^-(\tau_q)}$ happens for infinitely many $n$'s.
Since the random variables $s_n^-(\tau_q)$ are measurable with respect to $\mathcal{F}_{\tau_q}$, Blumenthal's 0-1 law implies that $\mathbb{P}( F_- \mid \mathcal{F}_{\tau_q}) $ is equal to 0 or 1 on the event $\{\tau_q \in \overline{A}_f(M)\}$.
 On the event $\{\tau_q \in \overline{A}_f(M)\}$, for every $n$, we have
\[
\mathbb{P}\Big(B_q(s_n^-(\tau_q)) \geq M\sqrt{s_n^-(\tau_q)} 
\mid
 \mathcal{F}_{\tau_q}\Big) = \mathbb{P}(B_q(1) \geq M
\mid
 \mathcal{F}_{\tau_q}) > 0.
\]
Since the right hand side does not depend on $n$, 
 by Remark \ref{6.30.1} (i),
$\mathbb{P}( F_- \mid \mathcal{F}_{\tau_q}) =1 $ on the event $\{\tau_q \in \overline{A}_f(M)\}$. Similarly, if $F_+$ denotes the event that $B_q(s_n^+(\tau_q)) \leq - M\sqrt{s_n^+(\tau_q)}$ happens for infinitely many $n$'s then $\mathbb{P}( F_+ \mid \mathcal{F}_{\tau_q}) =1 $ on the event $\{\tau_q \in \overline{A}_f(M)\}$.
 By the definition of the sequences $(s_n^-(t))$ and $(s_n^+(t))$, 
if $F_-\cup F_+$ holds then
 $\tau_q$ is not an isolated zero from the right. Therefore, the probability that $\tau_q \in \overline{A}_f(M)$ and that $\tau_q$ is an isolated point of $\zero(B-f)$ is equal to zero. Taking the union over all rational $q$'s and observing that $(A_f^- \cup A_f^+)^c = \bigcup_{M=1}^\infty\overline{A}_f(M)$ proves the claim.
\end{proof}

\begin{remark}\label{rem:basic_facts_extended}
By Proposition \ref{prop:basic_facts} (ii) and the time reversal property of Brownian motion it follows that, almost surely all isolated points of $\mathcal{Z}(B-f)$ are contained in the set
\[
\Big\{t\in \mathbb{R}^+: \lim_{h \uparrow 0}\frac{f(t+h)-f(t)}{\sqrt{|h|}} = \infty\Big\} \cup
\Big\{t\in \mathbb{R}^+: \lim_{h \uparrow 0}\frac{f(t+h)-f(t)}{\sqrt{|h|}} = -\infty\Big\}.
\]
Every point $t$ for which the limits
\[\lim_{h \downarrow 0}\frac{f(t+h)-f(t)}{\sqrt{h}} =  \lim_{h \uparrow 0}\frac{f(t+h)-f(t)}{\sqrt{|h|}}\]
are equal to $\infty$ or $- \infty$ is a strict local minimum or maximum. Since a function can have only countable many strict local extrema, almost surely all isolated points of $\mathcal{Z}(B-f)$ are contained in the set
\[
\Big\{t\in \mathbb{R}^+: \lim_{h \to 0}\frac{f(t+h)-f(t)}{\sqrt{|h|}\sign(h)} = \infty\Big\} \cup
\Big\{t\in \mathbb{R}^+: \lim_{h \to 0}\frac{f(t+h)-f(t)}{\sqrt{|h|}\sign(h)} = -\infty\Big\}.
\]
In particular, all isolated points of $\mathcal{Z}(B-f)$ are contained in the set of points of increase or decrease of $f$ (recall that $t$ is a point of increase of $f$ if for some $\epsilon > 0$ we have $f(s) < f(t)$ for $t-\epsilon < s < t$ and $f(t) <  f(s)$ for $t< s < t+\epsilon$ and points of decrease are defined analogously). Therefore, if $f$ is a function with at most countable many points of increase or decrease, then $\mathcal{Z}(B-f)$ has no isolated points almost surely. Examples of such functions include functions constructed by Loud in \cite{Loud} which satisfy a certain local reverse H\"{o}lder property at each point (see also the construction in \cite{MarxPiranian}).  These functions are defined as $g(t) = \sum_{k=1}^\infty g_k(t)$ where $g_k(t) = 2^{-2A\alpha k}g_0(2^{2Ak}t)$, for $0 < \alpha < 1$, a positive integer $A$ such that $2A(1-\alpha) > 1$, and a continuous function $g_0$ which has value $0$ at even integers, value $1$ at odd integers and is linear at all other points.
To prove that these functions have at most countably many points of increase or decrease, we proceed analogously as in the proof of the lower bound in \cite{Loud}. We will take $t \notin \mathbb{Q}$ and show that $t$ is not a point of increase. First observe that $\lfloor 2^{2Am}t\rfloor$ is odd for infinitely many integers $m$. For such an integer $m$ assume that $g_m(t) \geq 2^{-2A\alpha m-1}$ and denote $t_m = t+ 2^{-2A(m+1)}$. Then by construction $g_k(t) = g_k(t_m)$ for all $k > m$, $g_m(t_m) = g_m(t) - 2^{-2A(\alpha m + 1)}$ and $|g_k(t_m)-g_k(t)|  \leq 2^{2Ak(1-\alpha)-2A(m+1)}$, for $k < m$. Therefore 
\begin{align*}
g(t_m) -g(t)  & \leq - 2^{-2A(\alpha m + 1)} + 2^{-2A(m+1)}\sum_{k=1}^{m-1}2^{2Ak(1-\alpha)} \\
 &= \frac{2^{-2A(\alpha m + 1) + 1}  - 2^{-2A\alpha (m+1)} - 2^{-2A(m + \alpha)} }{2^{2A(1-\alpha)} - 1}.
\end{align*}
Using the fact that $2A(1-\alpha)>1$, it is easy to check that the right hand side above is negative, and since $t_m > t$ can be arbitrarily close to $t$ the claim follows. If $g_m(t) < 2^{-2A\alpha m - 1}$ then define $t_m = t - 2^{-2A(m+1)}$, which now satisfies  $g_m(t_m) = g_m(t) + 2^{-2A(\alpha m + 1)}$
and proceed analogously to prove that $g(t_m) > g(t)$.
\end{remark}

\begin{proof}[Proof of Proposition \ref{thm:Holder-1/2}]
This is straightforward from part (ii) of Proposition \ref{prop:basic_facts}.
\end{proof}




\section{Isolated zeros}\label{Isolated zeros - first example}


For $0 < \gamma < 1/2$, 
we will
 define the middle $(1-2\gamma)$-Cantor set and denote it by $C_\gamma$. 
Take
 a closed interval $I$ of length $|I|$. Define $\mathfrak{C}_{\gamma,1}$ as the set consisting of two disjoint closed subintervals of $I$ of length $\gamma |I|$, the left one (for which the left endpoint coincides with the left endpoint of $I$) and the right one (for which the right endpoint coincides with the right endpoint of $I$). 
Continue recursively, if $J \in \mathfrak{C}_{\gamma,n}$, then include in the set $\mathfrak{C}_{\gamma,n+1}$ its left and right closed subintervals of length $\gamma^{n+1}|I|$. Define the set $C_{\gamma,n}$ as the union of all the intervals from $\mathfrak{C}_{\gamma,n}$. 
For
 any $n$, the 
family
 $\mathfrak{C}_{\gamma,n}$ is the set 
of all
 connected components of the set $C_{\gamma,n}$.
The Cantor set is a compact set defined as
$C_\gamma=\bigcap_{n=1}^\infty C_{\gamma,n}$.
It is easy to show that $\dim(C_\gamma) = \log 2/\log (1/\gamma)$. 




Now we 
recall
 the construction of the standard Cantor function. Define the function $f_{\gamma,1}$ so that it has values $0$ and $1$ at the left and the right endpoint of the interval $I$, respectively, value $1/2$ on $I\backslash C_{\gamma,1}$ and interpolate linearly on the intervals in $\mathfrak{C}_{\gamma,1}$. Recursively, construct the function $f_{\gamma,n+1}$ so that for every interval $J =[s,t]\in \mathfrak{C}_{\gamma,n}$, the function $f_{\gamma,n+1}$ agrees with $f_{\gamma,n}$ at $s$ and $t$, it has value $(f_{\gamma,n}(s)+f_{\gamma,n}(t))/2$ on $J \backslash C_{\gamma,n+1}$ and interpolate linearly on the intervals in $\mathfrak{C}_{\gamma,n+1}$. See Figure \ref{fig:cantor}. 
It is easy to see that the sequence of functions $(f_{\gamma,n})$ converges uniformly on $I$.
We define the Cantor function $f_\gamma$ as the limit $f_\gamma=\lim_n f_{\gamma,n}$.
Note that for any $n$ and all $m \leq n$ the functions $f_\gamma$ and $f_{\gamma,n}$ agree at the endpoints of intervals $J \in \mathfrak{C}_{\gamma,m}$.

Another way to characterize the Cantor set $C_\gamma$ and the Cantor function $f_\gamma$ is by representing it as fixed points of certain transformations. Define linear bijections $g_{\gamma} \colon [0,\gamma] \to [0,1]$ and $h_{\gamma} \colon [1-\gamma,1] \to [0,1]$ by $g_{\gamma}(t) = t/\gamma$ and $h_{\gamma}(t)=(t-1+\gamma)/\gamma$. The Cantor set $C_\gamma$ defined on $[0,1]$ is the unique nonempty compact set that satisfies 
$C_\gamma = g_{\gamma}^{-1}(C_\gamma) \cup h_{\gamma}^{-1}(C_\gamma)$,
and the corresponding Cantor function is the unique continuous function that satisfies
\begin{equation}
\label{eq: cantor function definition}
f_{\gamma}(t) = \left\{\begin{array}{l l} f_{\gamma}(g_{\gamma}(t))/2, & 0 \leq t \leq \gamma,\\ 1/2, & \gamma \leq t \leq 1-\gamma, \\ 1/2 + f_{\gamma}(h_{\gamma}(t))/2, & 1-\gamma \leq t \leq 1.\ \end{array}\right.
\end{equation}


It is not hard to see that $f_\gamma$ is $\log 2 / \log (1/\gamma)$-H\"{o}lder continuous. The value $\gamma=1/4$ is the threshold at which the functions $f_\gamma$ become $1/2$-H\"{o}lder continuous.
For
 $\gamma < 1/4$ the function $f=f_\gamma$ will give an example for Theorem \ref{thm: isolated_zeros}. This threshold is sharp, since by Proposition \ref{thm:Holder-1/2}, for $\gamma \geq 1/4$ there are no isolated points in the zero set $\zero(B-f_{\gamma})$ almost surely.
For simplicity, we will assume the initial interval $I$ to be $[1,2]$, but we note that the analysis works for all compact intervals. To satisfy the assumptions of Theorem \ref{thm: isolated_zeros}, the function $f_{\gamma}$ should of course be extended to $\mathbb{R}^+ \backslash [1,2]$, by, say, value $0$ on $[0,1)$ and value $1$ on $(2,\infty)$.

\begin{proof}[Proof of Theorem \ref{prop: hitting characterisation}]
For an interval $I=[r,s] \in \mathfrak{C}_{\gamma.n}$, define $Z_n(I)$ as the event $B(s) \in [f_\gamma(r),f_\gamma(s)]$, and the random variable $Z_{\gamma,n} = \sum_{I \in \mathfrak{C}_{\gamma,n}}\mathbf{1}(Z_n(I))$, where $\mathbf{1}(Z_n(I))$ is the indicator function of the event $Z_n(I)$.
Note that there is a constant $c_1 > 0$, such that for any $0<\gamma <1/2$ we have 
\begin{equation}
\label{eq: hitting probabilities}
c_12^{-n} \leq \mathbb{P}(Z_n(I)) \leq 2^{-n} \ \text{ and } \ c_1 \leq \mathbb{E}(Z_{\gamma,n}) \leq 1.
\end{equation}
If $Z_{\gamma,n}>0$ happens for infinitely many $n$'s, then we can find a sequence of intervals $I_k=[r_k,s_k] \in \mathfrak{C}_{\gamma,n_k}$, such that $f_{\gamma}(r_k) \leq B(s_k) \leq f_{\gamma}(s_k)$, thus $|B(s_k) - f_\gamma(s_k)| \leq 2^{-n_k}$. Since $s_k \in C_{\gamma,n_k}$, the sequence $(s_k)$ will have a subsequence converging to some $s \in C_\gamma$, which obviously satisfies $B(s)=f_\gamma(s)$. 
Therefore
\begin{equation}
\label{eq: approximation of probabilities_1}
\mathbb{P}(\zero(B-f_\gamma) \cap C_{\gamma} \neq \emptyset) \geq \mathbb{P}(\limsup_{n\rightarrow \infty} \{Z_{\gamma,n}>0\}).
\end{equation} 

To estimate the probabilities $\mathbb{P}(Z_{\gamma,n}>0)$ from below we will use the inequality from Remark \ref{6.30.1} (ii), for which we need to bound the second moment $\mathbb{E}(Z_{\gamma,n}^2)$ from above. First we express it as 
\begin{equation}
\label{eq: second moment}
\mathbb{E}(Z_{\gamma,n}^2) = 2\sum_{I,J \in \mathfrak{C}_{\gamma,n} \atop I < J} \mathbb{P}(Z_n(I))\mathbb{P}(Z_n(J) \mid Z_n(I)) + \mathbb{E}(Z_{\gamma,n}).
\end{equation}
Now fix $n$ and intervals $I=[s_1,t_1]$ and $J=[s_2,t_2]$ from $\mathfrak{C}_{\gamma,n}$, so that $I<J$, and denote $a_i=f_\gamma(s_i)$ and $b_i=f_\gamma(t_i)$, for $i=1,2$. Let $\widetilde{B}$ be the process 
\[\widetilde{B}(t)=(t_2-t_1)^{-1/2}\Big(B(t_1+(t_2-t_1) t)-B(t_1)\Big),\]
which is, by the Markov property and Brownian scaling, again a Brownian motion, independent of $\mathcal{F}_{t_1}$, and thus independent of the event $Z_n(I) \in \mathcal{F}_{t_1}$. 
The
 event $Z_n(J)$ happens when $\widetilde{B}(1) \in \overline{J}$ for the interval $\overline{J}=[(t_2-t_1)^{-1/2}(a_2-B(t_1)) , (t_2-t_1)^{-1/2}(b_2-B(t_1))]$ of length $(t_2-t_1)^{-1/2}2^{-n}$.

\textbf{Case $\gamma < 1/4$:} 
Fix intervals $I^0$ and $J^0$
in
 $\mathfrak{C}_{\gamma,\ell+1}$, which are contained in a single interval 
in
 $\mathfrak{C}_{\gamma,\ell}$. Label the intervals from $\mathfrak{C}_{\gamma,n}$ contained in $I^0$ by $I_1, \dots, I_{2^{n-\ell-1}}$, and those contained in $J^0$ by $J_1, \dots, J_{2^{n-\ell-1}}$, so that $I_{i+1} < I_i$ and $J_j < J_{j+1}$. Set $I=I_i$ and $J=J_j$ for some $1 \leq i,j \leq 2^{n-\ell-1}$, and define $a_i$, $b_i$, $\widetilde{B}$ and $\overline{J}$ as before. Conditional on $Z_n(I)$, the left endpoint of the interval $\overline{J}$ is at least $(a_2-b_1)(t_2-t_1)^{-1/2}$, and since $a_2-b_1 = (i+j-2)2^{-n}$ we have $\overline{J} \subset [(i+j-2)2^{-n}\gamma^{-\ell/2}, \infty)$. Because $\overline{J}$ has length at most $(1-2\gamma)^{-1/2}2^{-n}\gamma^{-\ell/2}$ we obtain
\begin{align*}
\mathbb{P}(Z_n(J_j) \mid Z_n(I_i)) &  = \mathbb{P}(\widetilde{B}(1) \in \overline{J} \mid Z_n(I))\\
& \leq \frac{2^{-n}\gamma^{-\ell/2}}{\sqrt{2\pi(1-2\gamma)}}\exp\Big(-\frac{((j+i-2)2^{-n}\gamma^{-\ell/2})^2}{2}\Big),
\end{align*}
which, by summing over $1 \leq i,j \leq 2^{n-\ell-1}$ gives
\begin{multline}
\label{eq: conditional probabilities estimates}
\sum_{1 \leq i,j \leq 2^{n-\ell-1}}\mathbb{P}(Z_n(J_j) \mid Z_n(I_i))
\\
\leq 1 + \frac{1}{\sqrt{2\pi(1-2\gamma)}}
\sum_{k=1}^{\infty}(k+1)2^{-n}\gamma^{-\ell/2}\exp\Big(-\frac{(k2^{-n}\gamma^{-\ell/2})^2}{2}\Big).
\end{multline}
Here we used the trivial bound for $i=j=1$. The sum on the right hand side can be written as 
\begin{equation*}
S(a)=a\sum_{k = 1}^\infty(k+1)\exp(-(ka)^2/2),
\end{equation*}
for $a=2^{-n}\gamma^{-\ell/2}$. Since $\exp(-t^2/2) \leq \exp(-t+1/2)$, we see that 
\[
S(a) \leq e^{1/2}a\sum_{k \geq 1}(k+1)e^{-ka} = e^{1/2}a^{-1}\Big(\frac{a}{1-e^{-a}}\Big)^2(2e^{-a}-e^{-2a}).
\]
Since $a\mapsto (\frac{a}{1-e^{-a}})^2(2e^{-a}-e^{-2a})$ is a bounded function on $\mathbb{R}^+$, (\ref{eq: conditional probabilities estimates}) implies that for any fixed $I^0$ and $J^0$ as above
\begin{equation}
\label{eq:second_moment_estimates_below_1/4_3}
\sum_{I,J \in \mathfrak{C}_{\gamma,n} \atop I \subset I^0, J \subset J^0} \mathbb{P}(Z_n(J) \mid Z_n(I)) \leq 1+c_2 2^n\gamma^{\ell/2},
\end{equation}
for some $c_2>0$.
Therefore, summing the inequality in (\ref{eq:second_moment_estimates_below_1/4_3}) over all $I^0$ and $J^0$ and $\ell=0, \dots , n-1$, and using it together with (\ref{eq: second moment}) and (\ref{eq: hitting probabilities}), we have
\[
\mathbb{E}(Z_{\gamma,n}^2) 
\leq 2^{-n+1} \sum_{\ell=0}^{n-1}\Big(2^\ell + 2^\ell c_2 2^n\gamma^{\ell/2}\Big) + 1 \leq 2\sum_{\ell=0}^n\Big(2^{-(n-\ell)} + c_2(2\sqrt{\gamma})^\ell\Big),
\]
which is
 bounded since $2\sqrt{\gamma}<1$.

Thus we have shown that, for a fixed $\gamma < 1/4$, the second moments $\mathbb{E}(Z_{\gamma,n}^2)$ are bounded from above. Now the lower bound in the second inequality in (\ref{eq: hitting probabilities}) and the inequality from Remark \ref{6.30.1} (ii) imply that $\mathbb{P}(Z_{\gamma,n}>0) \geq \mathbb{E}(Z_{\gamma,n})^2/\mathbb{E}(Z_{\gamma,n}^2)$ is bounded from below and the claim follows from (\ref{eq: approximation of probabilities_1}) and Remark \ref{6.30.1} (i).

\begin{figure}
\includegraphics{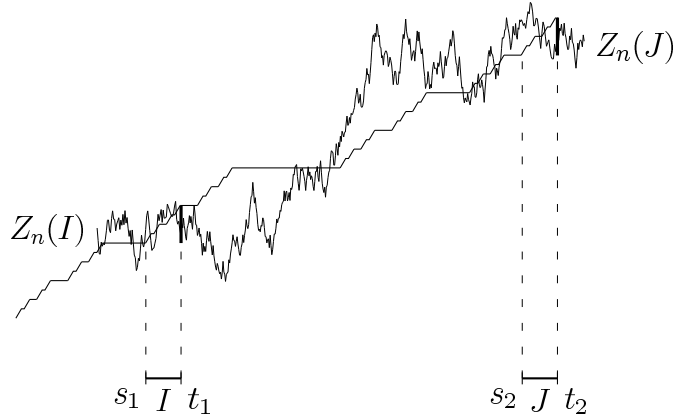}
\caption{Events $Z_n(I)$ and $Z_n(J)$ (graph of Brownian motion intersects two bold vertical intervals).}.
\label{fig:intersection_with_cantor}
\end{figure}

\textbf{Case $\gamma > 1/4$:} 
Again pick $I,J \in \mathfrak{C}_{\gamma,n}$ such that $[s_1,t_1]=I < J =[s_2,t_2]$ and define $a_i$, $b_i$, $\widetilde{B}$ and $\overline{J}$ as before. 
By $\ell$ denote the largest integer such that both $I$ and $J$ are contained in a single interval from $\mathfrak{C}_{\gamma,\ell}$.
Assume that $Z_n(I)$ happens. 
Clearly the endpoints of the interval $\overline{J}$ satisfy 
\[\frac{a_2-B(t_1)}{(t_2-t_1)^{1/2}} \geq 0 \ \text{ and } \  \frac{b_2-B(t_1)}{(t_2-t_1)^{1/2}} \leq \frac{1}{(1-2\gamma)^{1/2}(2\sqrt{\gamma})^{\ell}}.\]
The sequence $((2\sqrt{\gamma})^{-\ell})$ is bounded, and therefore the interval $\overline{J}$ is contained in a compact interval, which does not depend on the choice of $n$, $\ell$, $I$ or $J$. Using this and the fact that the length of $\overline{J}$ is bounded from above by $(1-2\gamma)^{-1/2}2^{-n}\gamma^{-\ell/2}$
and from below by $c' 2^{-n}\gamma^{-\ell/2}$,
 we get that 
for some positive constants $c_3$ and $c_4$ we have 
\begin{equation}
\label{eq: easy conditional probabilities}
c_32^{-n}\gamma^{-\ell/2} \leq \mathbb{P}(Z_n(J) \mid Z_n(I)) = \mathbb{P}(\widetilde{B}(1) \in \overline{J} \mid Z_n(I)) \leq c_4 2^{-n}\gamma^{-\ell/2}.
\end{equation}
Note that, since the sequence $((2\sqrt{\gamma})^{-\ell})$ is bounded for $\gamma=1/4$, estimates (\ref{eq: easy conditional probabilities}) also hold for $\gamma=1/4$.
Substituting
 (\ref{eq: easy conditional probabilities}) and the upper bounds from (\ref{eq: hitting probabilities}) into (\ref{eq: second moment}), and summing over all intervals $I$ and $J$, we obtain
\[
\mathbb{E}(Z_{\gamma,n}^2) \leq 1 + 2^{-n+1}\sum_{\ell=0}^{n-1}2^\ell2^{2(n-\ell-1)}c_4 2^{-n}\gamma^{-\ell/2} = 1 + \frac{c_4}{2}\sum_{\ell=0}^{n-1}(2\sqrt{\gamma})^{-\ell}.
\]
Since $2\sqrt{\gamma} > 1$, we have bounded $\mathbb{E}(Z_{\gamma,n}^2)$ from above by a constant not depending on $n$, and the claim follows as in the case $\gamma < 1/4$.

\textbf{Case $\gamma=1/4$: }
Assume that $\zero(B-f_\gamma)\cap C_\gamma \neq \emptyset$ and define $\tau$ as the first zero of $B-f_\gamma$ in the Cantor set $C_\gamma$ ($\tau$ exists since $\zero(B-f_\gamma)\cap C_\gamma$ is a closed set). For an interval $I = [s,t] \in \mathfrak{C}_{\gamma,n}$ assume that $\tau \in I$. Since $\tau$ is a stopping time, and by Brownian scaling, the conditional probability $\mathbb{P}\Big(Z_n(I)
\mid
 \mathcal{F}_\tau, \tau \in I\Big)$
is equal to
 the probability that Brownian motion at time $1$ is between $y_1=(f_\gamma(s)-f_\gamma(\tau))(t-\tau)^{-1/2}$ and $y_2=(f_\gamma(t)-f_\gamma(\tau))(t-\tau)^{-1/2}$. Since $f_\gamma(s) \leq f_\gamma(\tau) \leq f_\gamma(t)$ we see that $y_1 \leq 0$ and $y_2 \geq 0$. Moreover, the assumption $\gamma = 1/4$ implies that $t-\tau \leq 4^{-n} = (f_\gamma(t)-f_\gamma(s))^2$ which leads to $y_2-y_1 \geq 1$. 
Thus we can bound the probability 
\begin{equation}
\label{eq: approximation of probabilities_1-5}
\mathbb{P}\Big(Z_n(I)
\mid
 \mathcal{F}_\tau, \tau \in I \Big) 
\geq \mathbb{P}(0 \leq  B(1) \leq 1) = K^{-1},
\end{equation}
for some $K>0$, see also Figure \ref{fig: approximation from above}.
\begin{figure}
\includegraphics{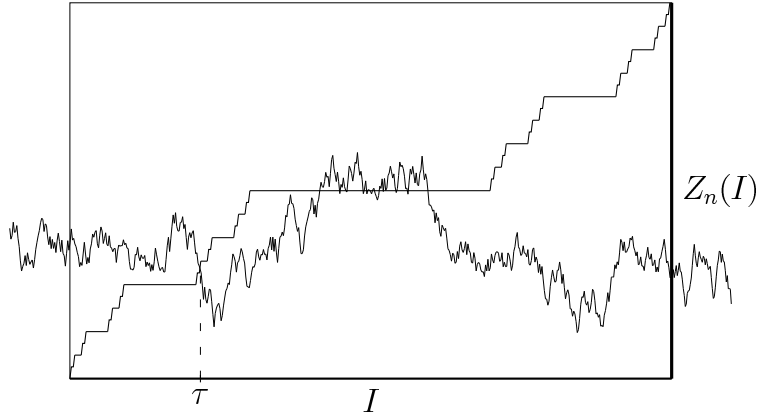}

\caption{If $\gamma \leq 1/4$, 
conditional
 on the event that there is a zero of $B-f_\gamma$ in an interval $I\in \mathfrak{C}_{\gamma,n}$ ($\tau$ is the first such zero) the probability of the event $Z_n(I)$ (Brownian motion intersecting the right hand side of the rectangle) is bounded from below.}
\label{fig: approximation from above}
\end{figure}

%
%
%
%
%
%
%
%
%
%
%
%
%
Therefore, $\mathbb{P}(Z_{\gamma,n}>0 \mid \tau \in I) \geq K^{-1}$ and, since the events $\{\tau \in I\}$ are disjoint for different $I \in \mathfrak{C}_{\gamma,n}$, we have $\mathbb{P}(Z_{\gamma,n}>0 \mid \zero(B-f)\cap C_\gamma \neq \emptyset) \geq K^{-1}$ for every $n$. Thus we obtain
\begin{equation}
\label{eq: approximation of probabilities_2}
\mathbb{P}(\zero(B-f_\gamma) \cap C_{\gamma} \neq \emptyset) \leq  K 
\inf_n
 \mathbb{P}(Z_{\gamma,n}>0).
\end{equation}
The arguments leading to (\ref{eq: approximation of probabilities_1-5}) are true for all $\gamma \leq 1/4$ and therefore (\ref{eq: approximation of probabilities_2}) holds for $\gamma \leq 1/4$.

For $I=[s,t] \in \mathfrak{C}_{\gamma,n}$ and $0 \leq \ell < n$, let $I^\ell$ denote the interval from $\mathfrak{C}_{\gamma,\ell}$ that contains $I$, and let $I_1^\ell, I_2^\ell \in \mathfrak{C}_{\gamma,\ell+1}$ be the left and right subintervals of $I^\ell$, respectively. 
Let $s=1.a_1 a_2 \dots a_n$ denote the $4$-ary expansion of the left endpoint of $I$ (note that the $4$-ary expansion of $s$ contains only the digits $0$ and $3$ and has length at most $n$, here we add zeros at the end, if necessary, to make it of length $n$). It is easy to see that $I \subset I_1^\ell$ if $a_{\ell+1}=0$ and $I \subset I_2^\ell$ if $a_{\ell+1}=3$. 
Call
an
 interval $I$
{\em balanced}
 if the sequence $a_1, \dots, a_n$ contains at least $n/3$ zeros and otherwise unbalanced. For
a
 balanced interval $I=[s,t] \in \mathfrak{C}_{\gamma,n}$ let $A_I$ denote the event that $I$ is the
leftmost
 balanced interval for which $Z_n(I)$ happens,
and, as before, let $1.a_1 \dots  a_n$ denote the $4$-ary expansion of $s$.
Assume that for some $0 \leq \ell < n$ we have $a_{\ell+1}=0$ and pick an interval $J \in \mathfrak{C}_{\gamma,n}$ such that $J \subset I_2^\ell$. 
Since $A_I \in \mathcal{F}_{t_1}$ we can use the same arguments that lead to the lower bound in (\ref{eq: easy conditional probabilities}) to conclude that the probability $\mathbb{P}(Z_n(J) \mid A_I) \geq c_5 2^{-(n-\ell)}$, for some constant $c_5>0$. Summing over all $J \subset I_2^\ell$ and over all $\ell$ such that $a_{\ell+1}=0$ gives 
\[\mathbb{E} (Z_{\gamma,n} \mid A_I) \geq c_5|\{1 \leq \ell \leq n : a_\ell =0\}| \geq c_5n/3.\]
By (\ref{eq: hitting probabilities}) and since events $A_I$ are disjoint, we have
\begin{multline}
\label{eq: 1/4 case balanced part}
\mathbb{P}(Z_n(I) \text{ for some balanced interval } I)\\ \leq  \frac{\mathbb{E}(Z_{\gamma,n})}{\mathbb{E}(Z_{\gamma,n} \mid Z_n(I) \text{ for some balanced interval } I)} \leq \frac{3}{c_5n}.
\end{multline}
To estimate the probability that $Z_n(I)$ happens for some unbalanced interval $I$
 notice that the number of such intervals is bounded from above by $e^{-c_6 n}2^n$ for some $c_6>0$. By 
(\ref{eq: hitting probabilities})  this gives
\begin{equation}
\label{eq: 1/4 case unbalanced part}
\mathbb{P}(Z_n(I) \text{ for some unbalanced interval } I) \leq e^{-c_6n}.
\end{equation}
Now (\ref{eq: 1/4 case balanced part}) and (\ref{eq: 1/4 case unbalanced part}) yield $\lim_{n \to \infty}\mathbb{P}(Z_{\gamma,n}>0)=0$ and the claim follows from (\ref{eq: approximation of probabilities_2}).
\end{proof}

Since the Cantor function $f_\gamma$ is $\log 2 / \log (1/\gamma)$-H\"{o}lder continuous, the following proposition proves Theorem \ref{thm: isolated_zeros}. 

\begin{proposition}
\label{prop:isolated_zeros_first_example}
For $\gamma<1/4$ consider the Cantor function $f_\gamma$. Then the set $\zero(B-f_\gamma)$ has isolated points with positive probability.
\end{proposition}

\begin{proof}
For $A \subset \mathbb{R}^+$ define $Z(A)$ as the event $\{\zero(B-f_\gamma) \cap A \neq \emptyset\}$.
We claim that there exists a constant $c_1$, such that for any interval $J \subset [0,1]$ of length $|J|$, we have
\begin{equation}
\label{eq:cuts_probabilities}
\mathbb{P}(Z(C_\gamma \cap f_\gamma^{-1}(J))) \leq c_1 |J|.
\end{equation} 
To prove this, fix an interval $J$ and take the largest integer $n$ such that $|J| \leq 2^{-n}$. Notice that $J$ can be covered by no more than two consecutive binary intervals $J_1$ and $J_2$ of length $2^{-n}$.
Moreover, there are consecutive $I_1,I_2 \in \mathfrak{C}_{\gamma,n}$ such that $f_\gamma(I_i) = J_i$ for  $i=1,2$,
and $C_\gamma \cap  f_\gamma^{-1}(J) \subset I_1 \cup I_2$, see Figure \ref{fig: cantor probabilities}.  Now using the notation from the proof of Theorem \ref{prop: hitting characterisation} and the arguments that lead to (\ref{eq: approximation of probabilities_1-5}) we obtain $\mathbb{P}(Z_n(I_i) \mid Z(C_\gamma \cap I_i)) \geq K^{-1}$ which yields
\begin{equation}
\label{eq:approximation_by_diagonals_for_small_gamma}
\mathbb{P}(Z(C_\gamma \cap I_i)) \leq K \mathbb{P}(Z_n(I_i)).
\end{equation} 
But by the first inequality in (\ref{eq: hitting probabilities}), the probability on the right hand side is bounded from above by $2^{-n}$. Using this fact in (\ref{eq:approximation_by_diagonals_for_small_gamma}) and summing the expression for $i=1,2$, we
obtain
 (\ref{eq:cuts_probabilities}).

By Theorem \ref{prop: hitting characterisation} the set $\zero(B-f_\gamma) \cap C_\gamma$ is non-empty with some probability $p>0$.
Take an arbitrary $\gamma < \gamma_1 < 1/4$ and $n_0$ such that $\sum_{n\geq n_0} (2\sqrt{\gamma_1})^{n} \leq p/(2c_1)$.

 For $n \geq n_0$ consider the interval $J_{k,n}=[k2^{-n}-\gamma_1^{n/2}/2, k2^{-n}+\gamma_1^{n/2}/2]$ and define the set $M_{n_0}= \bigcup_{n \geq n_0} \bigcup_{0 \leq k \leq 2^{n}}J_{k,n}$. By (\ref{eq:cuts_probabilities}) and the choice of $n_0$, we have that  $\mathbb{P}(Z(C_\gamma \cap f_\gamma^{-1}(M_{n_0}))) \leq p/2$. Therefore, the probability that there is a zero of $B(t)-f_{\gamma}(t)$ in the set $C_\gamma \cap \text{Int}(C_{\gamma,n_0})\backslash f_\gamma^{-1}(M_{n_0})$ is at least $p/2$ (here $\text{Int}(C_{\gamma,n_0})$ is the interior of the set $C_{\gamma,n_0}$). Now the claim will be proven if we show that any such zero is isolated. Take $t \in C_\gamma \cap \text{Int}(C_{\gamma,n_0})\backslash f_\gamma^{-1}(M_{n_0})$ and any $s \ne t$ in the same connected component of $\text{Int}(C_{\gamma,n_0})$. The largest integer $\ell$ such that both $s$ and $t$ are contained in the same interval of $\mathfrak{C}_{\gamma,\ell}$ satisfies $\ell \geq n_0$. Moreover, $|f_\gamma(s)-f_\gamma(t)| \geq \gamma_1^{(\ell+1)/2}/2$ and $|s-t| \leq \gamma^{\ell}$. Now it is clear that $t$ satisfies the condition in part (i) of Proposition \ref{prop:basic_facts} with $\alpha = \log \gamma_1 / (2 \log \gamma) < 1/2$. Therefore, the statement follows from part (i) of Proposition \ref{prop:basic_facts}.
\end{proof}

\begin{figure}
\includegraphics{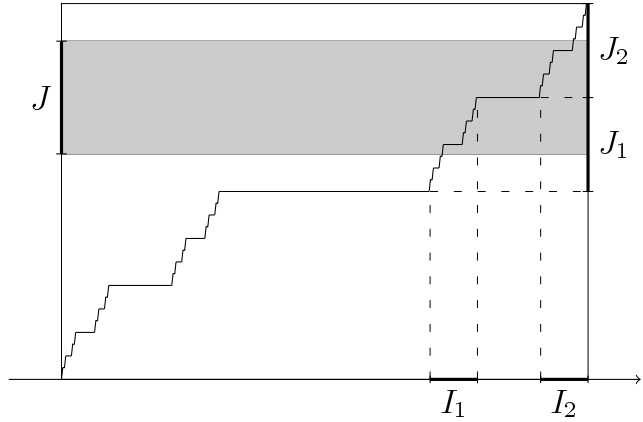}
\caption{An interval $J \subset [0,1]$ can be covered by two dyadic intervals $J_1$ and $J_2$ of comparable size. Intervals $I_1, I_2 \in \mathfrak{C}_{\gamma,n}$ are such that $f_{\gamma}(I_i) = J_i$, for $i=1,2$.}
\label{fig: cantor probabilities}

\end{figure}

\begin{remark}
\label{rem: isolated zeros almost surely}
It is not difficult to construct a continuous function $f$ such that the set $\zero(B-f)$ has isolated points almost surely. Let $f_\gamma$ be the Cantor function with $\gamma < 1/4$ defined on the interval $[0,1]$. Construct the function $f \colon \mathbb{R}^+ \to \mathbb{R}$ such that for every $n \geq 1$ and $0 \leq t \leq 1$ we have $f(4^{-n}(1+3t))=2^{-n}(1+f_\gamma(t))$ and define $f$ on $(1,\infty)$ arbitrarily, see Figure \ref{eq: repeat cantor}. 
By Proposition \ref{prop:isolated_zeros_first_example} and the Cameron-Martin theorem, the probability that $\zero(B-f)$ has an isolated point in the interval $[1/4,1]$ is positive, denote it by $p$. By Brownian scaling the probability that there is an isolated point of $\zero(B-f)$ in the interval $[4^{-n-1},4^{-n}]$ is also equal to $p$, for any $n \geq 1$.
Therefore, in view of Remark \ref{6.30.1} (i), the probability that there is an isolated point  of $\zero(B-f)$ in the interval $[4^{-n-1},4^{-n}]$, for infinitely many $n$'s is bounded from below by $p$. By Blumenthal's zero-one law this event has probability one, which proves the claim.
\end{remark}

\begin{figure}
\includegraphics{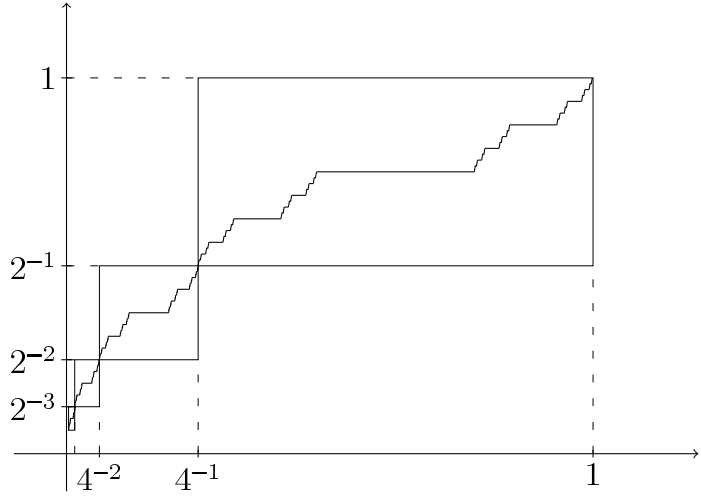}
\caption{Construction of the function $f$ from Remark \ref{rem: isolated zeros almost surely}. On each interval $[4^{-n-1},4^{-n}]$ function $f$ is a scaled and shifted copy of the Cantor function  (on the interval $(1,\infty)$ the function is defined arbitrarily).}
\label{eq: repeat cantor}
\end{figure}

The following proposition shows that, with positive probability, the set $\zero(B-f)$ can have only isolated points, or even only one point.

\begin{proposition}
\label{prop: only one zero}
There exists a continuous function $f$ such that, with positive probability, $\zero(B-f)$ is a singleton. 
\end{proposition}
\begin{proof}
Take $f_\gamma$ to be the Cantor function with $\gamma < 1/4$ defined on the interval $[1,2]$. 
Since $\zero(B-f_\gamma)$ has isolated points with positive probability, there are two rational numbers $q_1 < q_2$
such that,
with positive probability, there will be only one zero of $B-f_\gamma$ in the interval $(q_1,q_2)$. Denote this event by $D$ and on this event the unique zero by $\tau$. 
Note that on the event
$D\cap \{(B(q_1) - f_\gamma(q_1))(B(q_2) - f_\gamma(q_2))>0\}$ 
the unique zero $\tau$ is necessarily a local extremum of the process $B-f_\gamma$ and by part (i) of the upcoming Proposition \ref{prop: zeros extremes and record times} this event has  probability zero. Furthermore, on the event $D\cap \{B(q_1) < f_\gamma(q_1)\} \cap \{B(q_2)>f_\gamma(q_2)\}$ the unique zero $\tau$ is necessarily a local point of increase of the Brownian motion. 
By a result of Dvoretzky, Erd\"{o}s and Kakutani in \cite{DEK61}, almost surely, Brownian motion has no points of increase 
and thus 
\[\mathbb{P}(D,B(q_1) < f_\gamma(q_1),B(q_2)>f_\gamma(q_2)) =0,\] 
see also Theorem 5.14 in \cite{MP}.
Next define 
\[S=\{y>f_\gamma(q_1): \mathbb{P}(D\mid B(q_1) =y)>0\},\] 
and notice that by the Markov property and the discussion above $\mathbb{P}(B(q_1) \in S) > 0$. This implies that the set $S$ is of positive Lebesgue measure and so is $S_1=S \cap (f_\gamma(q_1)+\epsilon,\infty)$, for $\epsilon$ small enough. 
Now the claim will follow if we prove that there is a modification $f$ of the function $f_\gamma$ on the intervals $(0,q_1)$ and $(q_2,\infty)$, such that 
\begin{itemize}
\item[1)] with positive probability there are no zeros of $B-f$ in $(0,q_1)$ and $B(q_1)\in S_1$, 
\item[2)] for any $y< f_\gamma(q_2)-\epsilon$, conditional on $B(q_2)=y$, with positive probability there are no zeros of $B-f$ in $(q_2,\infty)$.
\end{itemize}

For 1) define $f$ to be linear on $[0,q_1]$ with $f(0)=-\epsilon$ and $f(q_1)=f_\gamma(q_1)$, if we do not require $f(0)=0$. To prove that the probability that both 
$\{\zero(B-f)\cap(0,q_1) = \emptyset\}$ and $\{B(q_1) \in S_1\}$ happen is positive, by the Cameron-Martin theorem, it is enough to prove that
\[\mathbb{P}\Big(\min_{t \in [0,q_1]}B(t) > -\epsilon, B(q_1) \in S_2=S_1-f_\gamma(q_1)-\epsilon\Big)>0.\] 
While this is intuitively obvious it can be proven by using the reflection principle at the first time the Brownian motion hits the level $-\epsilon$ to conclude that the probability that both $\min_{0 < t < q_1}B(t) \leq -\epsilon$ and $B(q_1)\in S_2$ happen is equal to the probability that $B(q_1) \in S_3$, where $S_3$ is obtained by reflecting the set $S_2$ around $-\epsilon$. Since $S_2 \subset \mathbb{R}^+$ we have $\mathbb{P}(B(q_1) \in S_3) < \mathbb{P}(B(q_1) \in S_2)$ and therefore 
\[
\mathbb{P}\Big(\min_{t \in [0,q_1]}B(t) > -\epsilon, B(q_1) \in S_2\Big) = \mathbb{P}(B(q_1) \in S_2) - \mathbb{P}(B(q_1) \in S_3) > 0.
\]
Now the probability that there is a unique zero in the interval $(0,q_2)$ is equal to some $p>0$.
If we do require $f(0)=0$, redefine $f$ on the interval $(0,\delta)$ as $f(t) = -c t^{1/3}$. Here $c>0$ is chosen large enough so that for $\delta>0$, for which $f$ is continuous, we have $\mathbb{P}(\zero(B-f)\cap (0,\delta)) < p$.


To satisfy the condition
2), replace
 $f_\gamma$ on $(q_2,\infty)$ by a linear function of slope $1$, that is $f(q_2+t)-f_\gamma(q_2) = t$. To prove that this $f$ satisfies the required condition on $(q_2,\infty)$ it is enough to prove that, for standard Brownian motion $B$, with positive probability there are no zeros of the process $B(t)-t-\epsilon$.
This probability is equal to $1-e^{-2\epsilon}$ by (5.13), Sect. 3.5 in \cite{KS}.
 See Figure \ref{fig: only one zero}.

\end{proof}

\begin{figure}
\includegraphics{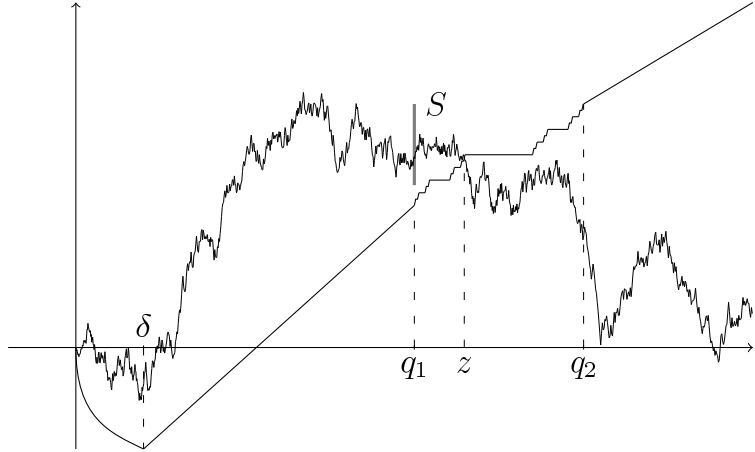}
\caption{Function $f$ from Proposition \ref{prop: only one zero}. It is a continuous function defined as $f(t)=-c t^{1/3}$ on $[0,\delta]$, linear on both $[\delta,q_1]$ and $[q_2,\infty)$ and is a part of the Cantor function on $[q_1,q_2]$. 
For this trajectory of $B$, there
 is only one zero of the process $B-f$, labelled as $z$.}
\label{fig: only one zero}
\end{figure}

The next proposition justifies Remark \ref{rem:no_isolated_extremes}.

\begin{proposition}
\label{prop: zeros extremes and record times}
Let $f \colon \mathbb{R}^+ \to \mathbb{R}$ be a continuous function and define the process $X(t) = B(t) - f(t)$. 
\begin{itemize}
\item[(i)] Almost surely there are no points in $\zero(X)$ which are local extrema.
\item[(ii)] Define $M_X(t)=\max_{0 \leq s \leq t}X(s)$ and the set of record times of the process $X$ as $\rec(X) = \{t>0: X(t) = M_X(t)\}$. Almost surely there are no isolated points in the set $\rec(X)$.
\end{itemize}
\end{proposition}

\begin{proof}
(i)
 Take an interval $[q_1,q_2]$ with $q_1>0$ and let $M$ be the maximum of the process $X$ on this interval. Then,
since the process $X$ has independent increments, $X(q_1)$ and $M-X(q_1)$ are independent. Since $X(q_1)$ has a continuous distribution, so has $M=(M-X(q_1))+X(q_1)$, and therefore $\mathbb{P}(M=x)=0$ for any $x \in \mathbb{R}$. Taking $x=0$ and a union over all rational $q_1 < q_2$ proves the claim for local maxima. Similarly the statement holds for local minima. See Figure \ref{fig: zeros and extrema}.

\begin{figure}
\includegraphics{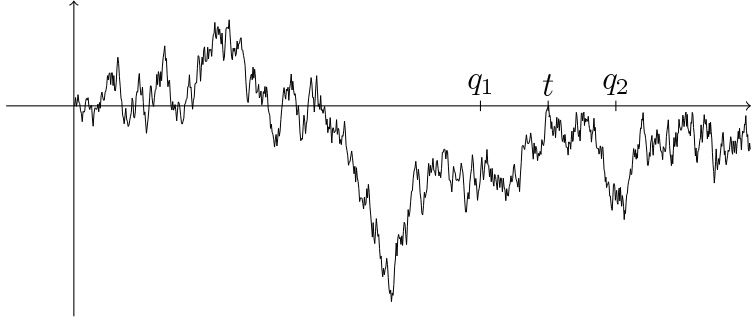}
\caption{Process $X(t)=B(t)-f(t)$ having a zero and a local extremum at some $t>0$; this is an event of probability zero.}
\label{fig: zeros and extrema}
\end{figure}

(ii)
 For any continuous function $g$, any record time $s \in \rec(g)$ is a maximum of $g$ on the interval $[s-\epsilon,s]$, for
every
 $\epsilon >0$.
Let $s > 0$ be an isolated point in $\rec(g)$. Then $s$ is a local maximum, because otherwise we would have record times
to the right of $s$,
 arbitrarily close to $s$. Since $g$ is continuous and there are no record times in the interval $(s-\epsilon,s)$ for some $\epsilon>0$, there has to be an $r \in \rec(g)\cup\{0\}$, which is also a local maximum
and such that $r<s$ and $g(r)=g(s)$.
Applying these observations to $g(t)=X(t)=B(t)-f(t)$, we see that, in order to prove the claim, it is enough to show that the process $X=B-f$ does not have two equal local maxima almost surely. See Figure \ref{fig: isolated record times}.
This is well known for standard Brownian motion and can be proven in the same way for the process $X$.
Namely, for two intervals $[q_1,r_1]$ and $[q_2,r_2]$, with $r_1 < q_2$, define the random variables $Y_1 = X(r_1) - \max_{q_1 \leq t \leq r_1}X(t)$ and $Y_3 = \max_{q_2 \leq t \leq r_2}X(t)-X(q_2)$, and let $Y_2=X(q_2)-X(r_1)$. Clearly these three random variables are independent.
Since $Y_2$ is a continuous random variable, so is $Y_1+Y_2+Y_3$ and
$\mathbb{P}(Y_1+Y_2+Y_3=0)=0$. Therefore, almost surely the maxima on $[q_1,r_1]$ and $[q_2,r_2]$ are different. Taking
the
 union over all possible rational $q_1$, $r_1$, $q_2$ and $r_2$ as above proves the claim.

\begin{figure}
\includegraphics{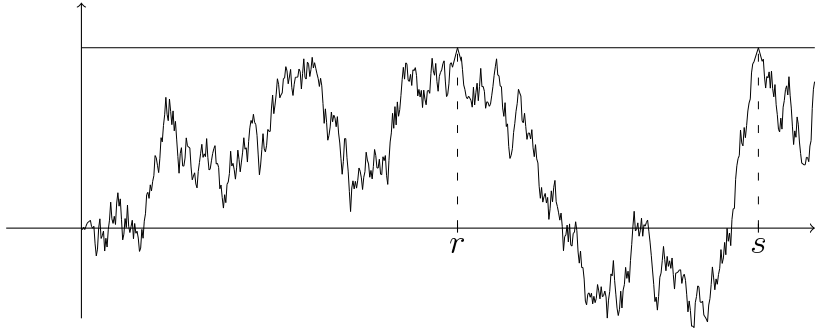}
\caption{
If $X(t)=B(t)-f(t)$ has an isolated record time at time $s$ then there exists a record time $r<s$ for which $X(r)=X(s)$.
 In this case both record times $r$ and $s$ have to be local maxima. This is an event of probability zero.}
\label{fig: isolated record times}
\end{figure}
\end{proof}

\section{On Hausdorff dimension of zero sets}
In this Section we prove Theorems \ref{set A_f has dim less than 1/2 }, \ref{thm:hausdorff_lower_bound} and \ref{thm:hausdorff_upper_bound}.

Recall that an interval $I$ is called \textit{dyadic} if it is of the form $I=[k2^{m},(k+1)2^m]$ for some integers $k>0$ and $m$. 
For a dyadic interval $I$, the set of its subintervals which are dyadic and of length $2^{-n}|I|$, will be denoted by $ \di_n(I)$.
For intervals $I$ and $J$ we will write $I < J$ if $J$ is located to the right of $I$.
 \begin{proof} [Proof of Theorem \ref{set A_f has dim less than 1/2 }]
Recall the definition of sets $A_f^+$ and $A_f^-$ from part (ii) of Proposition \ref{prop:basic_facts}. It  is enough to prove that both of these sets have Hausdorff dimension at most $1/2$.
This follows from the case $\alpha=1/2$ of part (i) of the following lemma, which estimates the Hausdorff dimension of larger sets.
\end{proof}

\begin{lemma}
\label{lemma:set_with_isolated_zeros}
(i)
 For a locally bounded function $f \colon \mathbb{R}^+ \to \mathbb{R}$ and $0 < \alpha < 1$ the sets
$B_f^+ = \{t\in \mathbb{R}^+: \liminf_{h \downarrow 0}\frac{f(t+h)-f(t)}{h^\alpha} > 0\}$  and  
$B_f^- = \{t\in \mathbb{R}^+: \limsup_{h \downarrow 0}\frac{f(t+h)-f(t)}{h^\alpha} < 0\}$
have Hausdorff dimension at most $\alpha$.

(ii)
 Assume $f \colon \mathbb{R}^+ \to \mathbb{R}$ has bounded variation on compact sets. For any $0 < \alpha < 1$ the set 
\[
B_f=\left\{ t\in \mathbb{R}^+ : \limsup_{h\downarrow 0} \frac{\left|f(t+h)-f(t)\right|}{h^\alpha} > 0 \right\}
\]
has Hausdorff dimension at most $\alpha$.
\end{lemma}

\begin{proof}
(i)
 Since $B_f^- = B_{-f}^+$ it is enough to prove the claim for $B_f^+$. First cover the set $B_f^+$  by sets $B_{f,n}^+ = \{t \in \mathbb{R}^+ :  f(t+h)-f(t) \geq h^\alpha/n, \text{ for all } 0 < h \leq 2^{-n}\}$, that is $B_f^+ = 
 \bigcup_{n=1}^\infty 
 B_{f,n}^+$. Note that for any positive integer $k$, and any $k$ points $t_1 < t_2 < \dots < t_k$ from $B_{f,n}^+$ such that $t_k \leq t_1 + 2^{-n}$, we have 
\[
f(t_k) \geq f(t_{k-1}) + \frac{(t_k-t_{k-1})^\alpha}{n} \geq  \dots \geq f(t_1) + \sum_{i=1}^{k-1}\frac{(t_{i+1}-t_i)^\alpha}{n}.
\]
Therefore, for any dyadic interval $I \subset \mathbb{R}^+$ of length $2^{-n}$ there is a constant $c_1$ such that, for any $k$ points $t_1 < \dots < t_k$ from $I \cap B_{f,n}^+$, we have 
\begin{equation}
\label{eq: hausdorff estimates}
\sum_{i=1}^{k-1}(t_{i+1}-t_i)^\alpha \leq c_1.
\end{equation}
Now for any positive integer $m$ and any dyadic interval 
$J \in \di_m(I)$ 
 such that $B_{f,n}^+ \cap J \neq \emptyset$
define $r_J^- = \inf\{B_{f,n}^+ \cap J\}$ and $r_J^+ = \sup\{B_{f,n}^+ \cap J\}$. 
The family $\mathcal{A}_m =\{[r_J^-,r_J^+]: J \in \di_m(I), B_{f,n}^+ \cap J \neq \emptyset\}$ is a cover of the set $B_{f,n}^+ \cap I$ with intervals of diameter at most $2^{-n-m}$.
From (\ref{eq: hausdorff estimates}) it is easy to see that 
\[
\sum_{J \in \mathcal{A}_m}|J|^\alpha = \sum_{J \in \di_m(I) \atop B_{f,n}^+ \cap J \neq \emptyset}(r_J^+ - r_J^-)^\alpha \leq c_1.
\]
Since $m$ was arbitrary, by the definition of Hausdorff dimension we obtain $\dim(B_{f,n}^+ \cap I) \leq \alpha$, for any dyadic interval $I$ of length $2^{-n}$. Therefore $\dim(B_{f,n}^+) \leq \alpha$ and taking the union over all $n$ gives $\dim(B_f^+) \leq \alpha$.



(ii) 
It is enough to prove that for any $c>0$ the set 
\[
B_f(c)=\Big\{t \in \mathbb{R}^+ : \limsup_{h\downarrow 0} \frac{\left|f(t+h)-f(t)\right|}{h^\alpha} > c\Big\}
\]
has Hausdorff dimension bounded by
$\alpha$
 from above.
Representing $B_f$ as the union $B_f =
\bigcup_{n=1}^\infty
 B_f(1/n)$ will prove the claim.
Fix a dyadic interval $I$ and an $\epsilon >0$. For every $t \in B_f(c) \cap I$ we can find
$0 < h_t < \epsilon$ such that $\left|f(t+h_t)-f(t)\right| \geq ch_t^\alpha$. Denote the interval $[t-h_t,t+h_t]$ by $I_t$.
Clearly the family of intervals $\{I_t: t \in B_f(c) \cap I\}$ is a cover of the set $B_f(c) \cap I$.
By Besicovitch's covering theorem (see Theorem 2.6 in \cite{Mattila}) we can find an integer $k$, not depending on $\epsilon$, and a subcover which can be represented as a union of $k$ at most countable disjoint subfamilies $\{I_t: t \in S_i\}$. More precisely, there are sets $S_i \subset B_f(c) \cap I$, $1 \leq i \leq k$ such that
\[
B_f(c) \cap I \subset \bigcup_{j=1}^k \bigcup_{t \in S_j}I_t \ \text{ and } \ I_s \cap I_t =\emptyset, \ \text{for all} \ s,t \in S_i, s \neq t.
\]
Clearly for any $i$ we have
\[
\sum_{t \in S_i}|I_t|^\alpha = 2^\alpha \sum_{t \in S_i}h_t^\alpha \leq \frac{2^\alpha}{c}\sum_{t \in S_i}|f(t+h_t)-f(t)| \leq \frac{2^\alpha M}{c},
\]
where $M$ is the total variation of $f$ on the interval $I$.
The last inequality follows form the fact that the intervals $I_t$, $t \in S_i$ are disjoint. Now summing the above inequality over $1 \leq i \leq k$ we obtain 
\[
\sum_{i=1}^k \sum_{t \in S_i}|I_t|^\alpha \leq \frac{2^\alpha Mk}{c}.
\]
Since $\{I_t:t\in S_i, 1 \leq i \leq k\}$ is a cover of $B_f(c) \cap I$ of diameter at most $2\epsilon$ and neither $k$ nor $M$ depend on $\epsilon$, the $\alpha$-dimensional Hausdorff measure of $B_f(c) \cap I$ is finite. Therefore $\dim(B_f(c) \cap I) \leq \alpha$ and, since the dyadic interval $I$ is arbitrary, the claim follows.
\end{proof}

The proof of Theorem \ref{thm:hausdorff_lower_bound} will be an application of the percolation method due to Hawkes \cite{Hawkes}, which we now describe. See Chapter 9 in \cite{MP} for more on this method.
Fix a dyadic interval $I$, a real number $0 < \beta < 1$ and set $p=2^{-\beta}$ (the construction and the results can be stated for any interval $I$).
Construct the set $\mathrm{S}_{\beta}(1)$ by including in it (as subsets) each of the two dyadic intervals from $\di_1(I)$ with probability $p$, independently of each other. To construct $\mathrm{S}_{\beta}(m+1)$, for each dyadic interval $J \in \di_m(I)$ such that $J \subset \mathrm{S}_{\beta}(m)$, include each of its dyadic subintervals from $\di_{m+1}(I)$ in $\mathrm{S}_{\beta}(m+1)$ with probability $p$, independently of each other. We obtain a decreasing sequence of compact sets $(\mathrm{S}_\beta(m))$ whose intersection we denote by $\Gamma[\beta]$.
Comparing this to the Galton-Watson process with binomial offspring distribution $B(2,p)$, we see  that $\mathbb{P}(\Gamma[\beta] \neq \emptyset) > 0$. The following theorem is due to Hawkes (Theorem 6 in \cite{Hawkes}).

\begin{theorem}[Hawkes]
For any set $A \subset I$, if $\mathbb{P}(A \cap \Gamma[\beta] \neq \emptyset )>0$ then $\dim(A) \geq \beta$.
\end{theorem}

In the above construction of the percolation set we can change the retention probabilities of intervals at each level. If $p_n = 2^{-\beta_n}$ is a retention probability at level $n$, we denote the union of intervals kept at level $m$ by $S_{(\beta_n)}(m)$,
and the limiting percolation set by $\Gamma[(\beta_n)]$. We will use the following result which is an easy corollary of Hawkes' theorem.

\begin{corollary}
\label{cor:Hawkes}
Let $0<\beta<1$ and let $(\beta_n)$ be a sequence converging to $\beta$. For any set $A \subset I$, if $\mathbb{P}(A \cap \Gamma[(\beta_n)] \neq \emptyset)>0$ then $\dim(A) \geq \beta$.
\end{corollary}

\begin{proof}
Let $\alpha < \beta$ and $m_0$ be such that $\alpha < \beta_m$ for all $m \geq  m_0$. There is a realization $B$ of $S_{(\beta_n)}(m_0)$ for which $\mathbb{P}(A \cap \Gamma[(\beta_n)]\neq \emptyset \mid S_{(\beta_n)}(m_0) = B)>0$. We have
\[\mathbb{P}(A \cap \Gamma[(\beta_n)] \neq \emptyset \mid S_{(\beta_n)}(m_0) = B) \leq \mathbb{P}(A \cap \Gamma[\alpha] \neq \emptyset \mid S_\alpha(m_0) = B),
\]
since we can couple two percolation processes so that for $m \geq m_0$, if an interval $J \in \di_m(I)$ is retained in the percolation process with retention probabilities $(2^{-\beta_n})$, it is also retained in the process with retention probability $2^{-\alpha}$. Since $\mathbb{P}(S_\alpha(m_0)=B)>0$ we get $\mathbb{P}(A \cap \Gamma[\alpha] \neq \emptyset)>0$ and, by Hawkes' theorem, $\dim(A) \geq \alpha$. Since $\alpha < \beta$ was arbitrary the claim follows.
\end{proof}


\begin{proof}[Proof of Theorem \ref{thm:hausdorff_lower_bound}]
It is enough to prove that the Hausdorff dimension of $\zero(B-f) \cap [1,2]$ is greater
than or equal to
 $1/2$ with positive probability.
Let $(\beta_n)$ be a sequence converging to $1/2$ from below, to be chosen later. Consider the percolation process on the interval $[1,2]$ with retention probabilities $(2^{-\beta_n})$, independent of Brownian motion.
Fix a positive integer $m$. 
For a dyadic interval $I \in \di_m([1,2])$ denote by $t_I$ its center and, for a fixed $\epsilon > 0$, consider the event
\[
\kz_{m,\eps}(I) = \{I \subset \mathrm{S}_{(\beta_n)}(m), |B(t_I) - f(t_I)| \leq \epsilon\}.
\]
Define $Y_{m,\epsilon} = \sum_{I \in \di_m([1,2])}\mathbf{1}(\kz_{m,\eps}(I))$, where $\mathbf{1}(\kz_{m,\eps}(I))$ is the indicator function of the event $\kz_{m,\eps}(I)$.
Using trivial bounds on the transition density of Brownian motion, the first moment of $Y_{m, \epsilon}$ can be estimated simply by
\begin{equation}
  \label{eq:first_moment_any_drift}
c_12^{m-\gamma_m}\epsilon \leq \mathbb{E}(Y_{m,\epsilon})  \leq 2\epsilon2^{m-\gamma_m},
\end{equation}
for some constant $c_1$ depending only on $\max_{t \in [1,2]}|f(t)|$,  and where $\gamma_m = \beta_1 + \dots + \beta_m$.
In the same way, for $I<J$ we can estimate the conditional probability
\begin{equation}
\label{eq: conditional hawkes}
\mathbb{P}\Big(|B(t_J)-f(t_J)| \leq \epsilon
\ \Big | \
 |B(t_I)-f(t_I)| \leq \epsilon\Big) \leq 2 \epsilon (t_J-t_I)^{-1/2}.
\end{equation}
For $I,J \in \di_m([1,2])$ such that $I < J$ let $\ell$ be the largest integer so that both $I$ and $J$ are contained in a single interval from $\di_\ell([1,2])$. In other words there are consecutive intervals $I^0, J^0 \in \di_{\ell+1}([1,2])$, contained in a single interval from $\di_\ell([1,2])$, such that $I \subset I^0$ and $J \subset J^0$.
Then 
\begin{equation}
\label{eq: conditional hawkes 2}
\mathbb{P}(I \cup J \subset \mathrm{S}_{(\beta_n)}(m)) = 2^{-2\gamma_m + \gamma_\ell}.
\end{equation} 
Using independence, (\ref{eq: conditional hawkes}) and (\ref{eq: conditional hawkes 2})
\begin{equation}
\label{eq:second_moment_help_any_drift}
\mathbb{P}(\kz_{m,\eps}(I) \cap \kz_{m,\eps}(J)) \leq (2\epsilon)^2 2^{-2\gamma_m + \gamma_\ell} (t_J-t_I)^{-1/2}.
\end{equation}
Summing (\ref{eq:second_moment_help_any_drift}) over all $I \subset I^0$ and $J \subset J^0$ for a fixed $I^0$ and $J^0$ as above
\begin{multline*}
\sum_{I \subset I^0, J \subset J^0}\mathbb{P}(\kz_{m,\eps}(I) \cap \kz_{m,\eps}(J)) \leq (2\epsilon)^2 2^{-2\gamma_m + \gamma_\ell} \sum_{k=1}^{2^{m-\ell}}k(k2^{-m})^{-1/2} \\
\leq  c_2 (2\epsilon)^2 2^{m/2-2\gamma_m + \gamma_\ell} 2^{3/2(m-\ell)} \leq 4c_2\epsilon^22^{2(m-\gamma_m) + \gamma_\ell-3\ell/2},
\end{multline*}
for some universal constant $c_2>0$.
Summing this over all $I^0$ and $J^0$ and all $0 \leq \ell \leq m-1$, and using (\ref{eq:first_moment_any_drift}), we can estimate the second moment
\begin{align}
  \label{eq:second_moment_any_drift}
 \mathbb{E}(Y_{m,\epsilon}^2)
& = 2\sum_{I,J \in \di_m([1,2]), I<J} \mathbb{P}(\kz_{m,\eps}(I) \cap \kz_{m,\eps}(J)) + \mathbb{E}(Y_{m,\epsilon}) \nonumber \\
&\leq 8c_2\epsilon^2\sum_{\ell=1}^{m-1} \Big(2^\ell 2^{2(m-\gamma_m) + \gamma_\ell-3\ell/2} \Big) +  2\epsilon 2^{m-\gamma_m} \nonumber \\
& \leq  8c_2 \epsilon^2 2^{2(m-\gamma_m)}\sum_{\ell=0}^{m-1} 2^{\gamma_\ell-\ell/2} + 2\epsilon 2^{m-\gamma_m}.
\end{align}
Now choose a sequence $(\beta_n)$ which converges to $1/2$ from below, and so that the series $\sum_{\ell=0}^{\infty} 2^{\gamma_\ell-\ell/2}$ converges (for example take $\beta_n = 1/2 - 2/(n \log 2)$, then $\gamma_\ell$ is up to an additive constant equal to $\ell/2 - 2\log_2 \ell$).
With such $(\beta_n)$ and for the sequence $\epsilon_m = 2^{-(m-\gamma_m)}$ that converges to zero, we define 
$\kzz_m = Y_{m,\epsilon_m}$. By (\ref{eq:first_moment_any_drift}) we have $\mathbb{E}(V_m) \geq c_1 > 0$ and by (\ref{eq:second_moment_any_drift}) we have $\mathbb{E}(V_m^2) \leq C_3 < \infty$, where $C_3$ is a universal constant.
Remark \ref{6.30.1} (ii) yields  $\mathbb{P}(\kzz_m > 0) \geq c_1^2/C_3$. Thus the probability of the event $\lim \sup_m\{\kzz_m > 0\}$ is also bounded from below by $c_1^2/C_3$. On this event there is a sequence $(s_m)$, such that each $s_m$ is the center of a dyadic interval from $\di_{k_m}([1,2])$ contained in $S_{(\beta_n)}(k_m)$, and such that $|B(s_m)-f(s_m)| \leq \epsilon_{k_m}$.
The sequence $(s_m)$ contains
 a subsequence that converges to some $s \in \Gamma[(\beta_n)]$ such that $B(s) = f(s)$. Therefore we have $\mathbb{P}(\zero(B-f) \cap \Gamma[(\beta_n)] \neq \emptyset) \geq c_1^2/C_3 > 0$.  By the independence of Brownian motion and the percolation process and Corollary \ref{cor:Hawkes} it follows that  
\begin{equation}\label{eq: Hausdorff uniform bound Holder 1/2}
\mathbb{P}(\dim(\zero(B-f)) \geq 1/2) \geq c_1^2/C_3 > 0.
\end{equation}

\end{proof}



\begin{remark}\label{rem: uniform bound for holder}
Note that, for a continuous function $f$ defined on $[1,2]$, the lower bound from \eqref{eq: Hausdorff uniform bound Holder 1/2} depends only on $\max_{t \in [1,2]}|f(t)|$.
\end{remark}

\begin{proof}[Proof of Theorem \ref{thm:hausdorff_upper_bound}]
It is enough to prove that $\dim(\zero(B-f)\cap I) \leq 1/2$, for any dyadic interval $I$.
First we will prove the claim for functions which are $1/2$-\holder on compact intervals.

Assume $f$ is a $1/2$-\holder continuous function on $I$, that is $|f(t)-f(s)| \leq c_0 |t-s|^{1/2}$ for some $c_0>0$ and all $s,t \in I$. 
For an interval $J=[s_1,s_2] \subset I$ set $Z(J) = 1$, if there exists $t \in J$ such that $B(t)=f(t)$, and $Z(J) = 0$ otherwise, and define the interval $\overline{J} = [f(s_1)-c_0 \sqrt{s_2-s_1},f(s_1)+c_0 \sqrt{s_2-s_1}]$.
On the event $Z(J)=1$ define the stopping time $\tau=\min\{\zero(B-f)\cap J\}$. Since $(\tau,B(\tau)) \in J \times \overline{J}$,
by the
 strong Markov property, conditional on the sigma algebra $\mathcal{F}_\tau$ and on the event $\{Z(J)=1\}$, the probability
$p_1$ that $B(s_2)\in \overline{J}$
 is equal to the probability that $B(1) \in \overline{J}_\tau$, where $\overline{J}_\tau$ is the interval $\overline{J}$ shifted by $-B(\tau)$ and scaled by $(s_2-\tau)^{-1/2}$. Since the interval $\overline{J}_\tau$ has length at least $2c_0$ and contains the origin,
$p_1$
 is bounded from below by a constant not depending on the choice of the interval $J$; see also arguments in the proof of Theorem \ref{prop: hitting characterisation} leading to (\ref{eq: approximation of probabilities_1-5}). Therefore, for some
$c_1 < \infty$
 we obtain
$\mathbb{P}(B(t) \in \overline{J}\mid Z(J)=1) \geq c_1^{-1}$. This implies 
\begin{equation}
\label{eq: hausdorff conditional}
\mathbb{P}(Z(J)=1) \leq c_1 \mathbb{P}(B(t) \in \overline{J})\leq c_2|J|^{1/2},
\end{equation} 
for some $c_2>0$.
%
%
%
%
%
%

Now consider the covering $\mathcal{A}_k$ of the set $\zero(B-f)\cap I$, consisting of the dyadic intervals from  $\di_k(I)$ which intersect the set $\zero(B-f)\cap I$. Since every interval in $\mathcal{A}_k$ has length $2^{-k}|I|$, by (\ref{eq: hausdorff conditional}) we have
\begin{multline*}
    \mathbb{E}\Bigg(\sum_{J \in \mathcal{A}_k} \left|J\right|^{\frac{1}{2}} \Bigg)= \mathbb{E}\Bigg( \sum_{J \in \di_k(I)} Z(J) 2^{-\frac{k}{2}}|I|^{1/2} \Bigg)  \\
\leq 2^k c_2 2^{-\frac{k}{2}}|I|^{1/2} 2^{-\frac{k}{2}}|I|^{1/2} = c_2|I|.
\end{multline*}
Fatou's lemma implies
\begin{align*}
     \mathbb{E}\Bigg( \liminf_{k\to \infty} \sum_{J \in \mathcal{A}_k} |J|^{1/2} \Bigg) \leq  \liminf_{k\rightarrow \infty} \mathbb{E}\Bigg( \sum_{J \in A_k} |J|^{1/2} \Bigg) \leq c_2|I|.
\end{align*}
 Therefore, almost surely, we can find a sequence of coverings $\{J^k_n, n\geq 1\}$ of $\zero(B-f)\cap I$, with $\lim_{k\to \infty}\sup_{n\geq 1} |J^k_n| =0$ and $\limsup_{k\to \infty}\sum_{n\geq 1} |J^k_n|^{1/2} <\infty$. This implies that $\dim(\zero(B-f) \cap I) \leq 1/2$, a.s.
  

%
%

Now let $f$ be of bounded variation on compact intervals and define the set 
\[B_f=\left\{ t\in \mathbb{R}^+ : \limsup_{h\downarrow 0} \left|f(t+h)-f(t)\right|h^{-1/2} \geq 1 \right\}.\]
By part (ii) of Lemma \ref{lemma:set_with_isolated_zeros} we have $\dim(B_{f}) \leq 1/2$.
Therefore, it is enough to bound the dimension of the zero set in the complement, $\zero(B-f) \cap B_f^c$, where $B_f^c = I \backslash B_f$.
We can cover $B_f^c$ by a countable union of the sets 
\begin{align*}
D_{f,n}= \left\{ t\in I : \left| f(t+h) - f(t)\right| \leq \sqrt{h}, \text{ for all } 0 \leq  h < 2^{-n} \right\}.
\end{align*}
For a fixed $n$, all $t_1, t_2 \in D_{f,n}$ with $|t_1 - t_2| < 2^{-n}$ satisfy $\left| f(t_2) - f(t_1)\right| \leq |t_1 - t_2|^{1/2}$. Therefore, the restriction $f|_{D_{f,n}}$ is $1/2$-\holder continuous, that is,
for some positive constant $c_3$ and all $t_1,t_2 \in D_{f,n}$ we have  $\left| f(t_2) - f(t_1)\right| \leq c_3 |t_1 - t_2|^{1/2}$. 
Define the function $f_n\colon I\to \mathbb{R}$ to be $f_n(t)=f(t)$ for all $t\in D_{f,n}$.
We will define $f_n$ for $t\in I \setminus  D_{f,n}$ using linear interpolation, in a sense. If  $D_{f,n} = \emptyset$ then set $f_n$ to be any constant function. Assume that  $D_{f,n} \ne \emptyset$.
Note
 that the set $D_{f,n}$ is closed from the left, that is if $(s_k)$ is an increasing sequence of points in $D_{f,n}$ converging to some $s$, then $s \in D_{f,n}$.
Since $f$ has bounded variation, by assumption, the right limit $\lim_{t \downarrow s}f(t)$ exists for all $s\in I$. 
 Thus if $t \in I \backslash D_{f,n}$ define $t_l = \max\{s \in D_{f,n}: s<t\}$, $t_r=\inf\{s \in D_{f,n}: t<s\}$ as well as $a_l=f(t_l)$ and $a_r=\lim_{s \downarrow t_r}f(s)$ and notice that $|a_r - a_l| \leq c_3\sqrt{t_r-t_l}$. Now define $f_n$ on $[t_l,t_r]$ to be linear with $f_n(t_l)=a_l$ and $f_n(t_r)=a_r$. If $t_l$ does not exist, define $t_l$ as the left endpoint of $I$ and $a_l = a_r$, and do similarly if $t_r$ does not exist. 
Clearly $f_n$ is $1/2$-H\"{o}lder continuous with the same constant $c_3$.
We can apply the first part of the proof to $f_{n}$, to conclude that the set $\zero(B-f_n)$ has Hausdorff dimension at most $1/2$ almost surely. Since $\zero(B-f) \cap D_{f,n} \subset \zero(B-f_n)$, we obtain $\dim \left(\zero(B-f)\cap D_{f,n} \right)\leq 1/2$ and $\dim \left(\zero(B-f)\cap \bigcup_{n\in \NN} D_{f,n} \right)\leq 1/2$. Since $B_f^c \subset \bigcup_{n \in \mathbb{N}}D_{f,n}$ the claim follows.
\end{proof}

We finish the paper with the proof of Corollary \ref{cor:sharp_haudorff_for_holder}.

\begin{proof}[Proof of Corollary \ref{cor:sharp_haudorff_for_holder}]
Assume that $|f(t)-f(s)| \leq C\sqrt{|t-s|}$ for all $0<s,t<1$.
For every positive integer $n$ define a continuous function $f_n \colon [1,2] \to \mathbb{R}$ as $f_n(t) = 2^{n/2}f(t2^{-n})$, for all $t\in [1,2]$. Clearly $\max_{t\in[1,2]}|f_n(t)| \leq C\sqrt{2}$ and, by Remark \ref{rem: uniform bound for holder} and Brownian scaling, there is a constant $c>0$ such that 
\[
\mathbb{P}(\dim(\zero(B-f)\cap [2^{-n},2^{-n+1}]) \geq 1/2) = \mathbb{P}(\dim(\zero(B-f_n)) \geq 1/2) \geq c.
\]
Thus by Blumenthal's zero-one law, almost surely, there are infinitely many integers $n$ such that the Hausdorff dimension of the set $\zero(B-f)\cap [2^{-n},2^{-n+1}]$ is greater or equal than $1/2$. The upper bound follows from Theorem \ref{thm:hausdorff_upper_bound}.
\end{proof}

\section*{Acknowledgments}

Authors would like to thank David Aldous, Zhen-Qing Chen, Serguei Denissov and Steven Evans for helpful discussions and references.
Ton\'{c}i Antunovi\'{c} and Julia Ruscher would like to thank Microsoft Research where this work was completed.
Krzysztof Burdzy was supported in part by NSF Grant DMS-0906743 and by grant N N201 397137, MNiSW, Poland.




\end{document}